\documentclass[11pt,fleqn]{amsart}

\usepackage{latexsym,amsmath,amssymb,amsfonts,MnSymbol,epsfig,verbatim}
\usepackage[initials]{amsrefs}
\usepackage[all]{xy}
\usepackage{tikz}
\usetikzlibrary{decorations.pathmorphing}

\providecommand{\CC}{{\mathbb{C}}}
\providecommand{\RR}{{\mathbb{R}}}
\providecommand{\QQ}{{\mathbb{Q}}}

\providecommand{\ZZ}{{\mathbb{Z}}}

\providecommand{\SO}{{\mathrm{SO}(n)}}
\providecommand{\Spin}{{\mathrm{Spin}(n)}}
\providecommand{\Spinc}{{\mathrm{Spin}^c(n)}}

\providecommand{\FSO}{{\mathcal{F}_{\mathrm{SO}}}}
\providecommand{\GL}{{\mathrm{GL}(n,\RR)}}

\providecommand{\F}{{\mathcal{F}}}

\providecommand{\Ind}{{\mathrm{Index}}}
\providecommand{\Kpt}{{K_0(\cdot)}}

\providecommand{\so}{{\mathfrak{so}}}

\newtheorem{theorem}{Theorem}
\newtheorem{lemma}[theorem]{Lemma}
\newtheorem{corollary}[theorem]{Corollary}
\newtheorem{proposition}[theorem]{Proposition}

\theoremstyle{definition}
\newtheorem{definition}[theorem]{Definition}

\theoremstyle{remark} 
\newtheorem{remark}[theorem]{Remark}

\numberwithin{equation}{section}

\begin{document}

\title{$K$-homology and Fredholm Operators I: Dirac Operators}
\author{Paul F.\ Baum}
\address{The Pennsylvania State University, University Park, PA, 16802, USA}
\email{baum@math.psu.edu}
\author{Erik van Erp}
\address{Dartmouth College, 6188, Kemeny Hall, Hanover, New Hampshire, 03755, USA}\email{jhamvanerp@gmail.com}

\thanks{PFB was partially supported by NSF grant DMS-0701184}
\thanks{EvE was partially supported by NSF grant DMS-1100570}

\maketitle

\tableofcontents

\section{Introduction}

In this expository paper we prove the following special case of the Atiyah-Singer index theorem \cite{AS1}.
\begin{theorem}\label{thm}
Let $M$ be an even dimensional compact Spin$^c$ manifold without boundary with Dirac operator $D$.
If $E$ is a smooth $\CC$ vector bundle on $M$ then
\[ \Ind\,(D_E) = (\mathrm{ch}(E)\cup \mathrm{Td}(M))[M]\]
Here  $D_E$ is $D$ twisted by $E$, $\mathrm{ch}(E)$ is the Chern character of $E$,
$\mathrm{Td}(M)$ the Todd class of the Spin$^c$ vector bundle $TM$,
and $[M]$ is the fundamental cycle of  $M$.
\end{theorem}

This expository paper is the first of three.
In the present paper we prove the index  theorem for Dirac operators.
In \cite{BvE2} we reduce the general elliptic operator case to the Dirac case.
Finally, in \cite{BvE3} we reduce the case of hypoelliptic operators on contact manifolds to the elliptic case.
The unifying theme of these papers is that $K$-homology provides the topological foundation for index theory.
The  $K$-homology point of view is that any reasonable index problem can be solved by reducing to the Dirac case.

The first two papers are spin-offs of the third \cite{BvE3}.
These papers were written to clear up basic points about index theory that are generally accepted as valid, but for which no proof has been published.
Some of these points are needed for the third paper in the series  \cite{BvE3}.

There are two kinds of proof of the Atiyah-Singer theorem.
The heat kernel proof essentially uses no topology \cite{BGV}.
In this proof all the calculations are done on the original manifold $M$. 
All other proofs use topology, and proceed by a reduction in steps to a case where the formula can be  verified by a straightforward calculation.

Our proof of Theorem \ref{thm} is organized in such a way as to clearly reveal the fundamental role of Bott periodicity.
The Axioms in the $K$-theory proof of \cite{AS1} are replaced by two moves: bordism and vector bundle modification.
The proof relies on the invariance of the topological and analytical index under these two moves.
Following the approach of Grothendieck \cite{BoSe} (who reduced to $\CC P^n$), we reduce to the problem on a sphere. 
Bott perodicity determines (up to stable isomorphism) all the $\CC$ vector bundles on spheres.
In the end one needs to do one index calculation.
We formalize this in the assertion that Bott Periodicity determines the geometrically defined $K$-homology of a point.
\footnote{Throughout this paper $K$-theory is Atiyah-Hirzebruch $K$-theory. In particular, this is $K$-theory with compact supports.}

Our aim in the present paper and its sequel \cite{BvE2} is to give the simplest possible exposition of the topological proof of Atiyah-Singer. 
The proof given here has the merit that it applies with no essential changes to the equivariant case and the families case.

\section{Dirac operators}

\subsection{Dirac operator of $\RR^n$}\label{Dirac_Rn}

To define the  Dirac operator of $\mathbb{R}^n$ we shall construct matrices $E_1, E_2, \hdots , E_n$ with the following properties:
\begin{itemize}
\item Each $E_j$ is a $2^r\times2^r $ matrix of complex numbers,\\
 where r is the largest integer $\le n/2$ (i.e. $n=2r$ or $n=2r+1$).
\item Each $E_j$ is skew adjoint, i.e. $E_j^* = -E_j$
\item $E_j^2 = -I,\quad j =1, 2, \dots, n$ \quad
($I $ is the $2^r\times2^r $ identity matrix.)
\item $E_jE_k + E_kE_j = 0$ whenever $j \ne k$.
\item For $n$ even each $E_j$ is of the form
\[
E_j=
\left[
\begin{matrix}
\mathbf{0}&*\\
*&\mathbf{0}
\end{matrix}
\right]
\]
and
\[ i^rE_1E_2\cdots E_n=
\left[
\begin{matrix}
I&\mathbf{0}\\
\mathbf{0}&-I
\end{matrix}
\right]
\qquad i = \sqrt{-1}
\]
\item For $n$ odd 
\[ i^{r+1}E_1E_2 \cdots E_n = I\]
\end{itemize}

These matrices are constructed by a simple inductive procedure.
If $n=1$ then  $E_1 =[-i]$.
New matrices $\widetilde{E}_1,\widetilde{E}_2, \ldots,\widetilde{E}_{n+1}$ are
constructed from $E_1, \ldots, E_n$ as follows.
If $n$ is odd:
\[
\widetilde{E}_j=
\left[
\begin{smallmatrix}
\mathbf{0}& E_j\\
E_j &\mathbf{0}
\end{smallmatrix}
\right] 
\quad\text{for }j=1,\ldots,n\quad\text{and}\quad \widetilde{E}_{n+1}=
\left[
\begin{smallmatrix}
\mathbf{0}&-I\\
I&\mathbf{0}
\end{smallmatrix}
\right]
\]
If $n$ is even:
\[
\widetilde{E}_j=E_j
\quad\text{for }j=1,\ldots,n\quad\text{and}\quad \widetilde{E}_{n+1}=
\left[
\begin{smallmatrix}
-iI&\mathbf{0}\\
\mathbf{0}&iI
\end{smallmatrix}
\right]
\]
For $n=1, 2, 3, 4$ the matrices are:
\begin{itemize}
\item[$n=1$] $E_1=[-i]$
\vskip 6pt
\item[$n=2$] $E_1=
\left[\begin{smallmatrix}
0&-i\\
-i&0
\end{smallmatrix}
\right]$,
$E_2=
\left[\begin{smallmatrix}
0&-1\\
1&0
\end{smallmatrix}
\right]$
\vskip 6pt
\item[$n=3$] $E_1=
\left[\begin{smallmatrix}
0&-i\\
-i&0
\end{smallmatrix}
\right]$,
$E_2=
\left[\begin{smallmatrix}
0&-1\\
1&0
\end{smallmatrix}
\right]$, 
$E_3=
\left[
\begin{smallmatrix}
-i&0\\
0&i
\end{smallmatrix}
\right]$
\vskip 6pt
\item[$n=4$] $E_1=
\left[
\begin{smallmatrix}
0&0&0&-i\\
0&0&-i&0\\
0&-i&0&0\\
-i&0&0&0
\end{smallmatrix}
\right]$
$E_2=
\left[
\begin{smallmatrix}
0\:&0&\:0& -1\\
0\:&0&\:1  &0\\
0\:&-1 \:&0&0\\
1\: &0&\:0&0  
 \end{smallmatrix}
\right]$\vspace{10mm}
$E_3=
\left[
\begin{smallmatrix}
0\:&0&-i\:&0\\
0\:&0&0\:&i\\
-i\:&0&0\:&0\\
0\:&i&0\:&0  
\end{smallmatrix}
\right]$  
$E_4=
\left[
\begin{smallmatrix}
0\:\:&0& -1&0\\
0\:\:&0&0& -1\\
1\:\:  &0&0&0\\
0\:\:&  1&0&0
\end{smallmatrix}
\right]$
\end{itemize}
The Dirac operator of $\RR^n$ is
\[ D = {\displaystyle \sum_{j=1}^n \ E_j\frac{\partial }{\partial x_j}}  \]
$D$ acts on smooth sections of the trivial complex vector bundle $\RR^n\times \CC^{2^r}$, i.e.
\[ D\;\colon\; C_c^\infty(\RR^n,\CC^{2^r})\to C_c^\infty(\RR^n,\CC^{2^r})\]

\subsection{Properties of the Dirac operator of $\RR^n$}

The operator $D$ is symmetric due to $E_j^*=-E_j$.
$D$ is essentially self-adjoint, i.e. the closure of $D$ is an unbounded self-adjoint operator on the Hilbert space
\[ L^2(\RR^n)\otimes \CC^{2^r} = L^2(\mathbb{R}^n) \oplus L^2(\mathbb{R}^n) \oplus \hdots \oplus L^2(\mathbb{R}^n) \] 
Because $E_j^2=-I$ and $E_jE_k+E_kE_j=0$ when $j\ne k$ the square of the Dirac operator is the Laplacian
\[ D^2= \Delta\otimes I_{2^r}\qquad \Delta=\sum_{j=1}^n - \frac{\partial^2}{\partial x_j^2}\]
If $n$ is even, the trivial bundle $\RR^n\times \CC^{2^r}$ is graded by the matrix 
$\left[
\begin{smallmatrix}
I&\mathbf{0}\\
\mathbf{0}&-I
\end{smallmatrix}
\right]$
and since each $E_j$ is of the form 
$\left[
\begin{smallmatrix}
\mathbf{0}&*\\
*&\mathbf{0}
\end{smallmatrix}
\right]$
the Dirac operator is odd, i.e. anticommutes with the grading operator.

$D$ is a first order elliptic constant coefficient differential operator.

\subsection{The structure group}\label{strgrp}

The Dirac operator $D$ of $\RR^n$ is {\em not} $\mathrm{SO}(n)$ equivariant,
but it is $\Spin$ and  $\Spinc$ equivariant.

For $n=1,2,3,\dots$ the Clifford algebra $C_n$ is the universal unital $\RR$ algebra with $n$ generators $e_1,\dots,e_n$ and relations
\[ e_i^2=-1,\qquad e_ie_j=-e_je_i,\quad i\ne j\]
The dimension of $C_n$ is $2^n$, and the set of elements $e_{i_1}e_{i_2}\cdots e_{i_p}$, $i_1<i_2<\dots<i_p$ is a basis (including the element $1\in C_n$ for $p=0$).

Let $\Lambda^p\subset C_n$ denote the subspace spanned by the basis elements $e_{i_1}e_{i_2}\cdots e_{i_p}$, with $\Lambda^0=\RR$ spanned by $1\in C_n$.
$\Lambda^2$ is closed under commutators in $C_n$.
The Lie algebra $\Lambda^2$  is represented on $\Lambda^1=\RR^n$,
\begin{align*}
 d\rho\; &\colon\; \Lambda^2\to \mathrm{End}(\RR^n)\\
d\rho(\alpha).v &:= \alpha v- v\alpha \qquad \alpha\in \Lambda^2, \;v\in \Lambda^1=\RR^n
\end{align*}
The Lie algebra $\so(n)$ of $\SO$ consists  of skew symmetric $n\times n$ matrices with real coefficients. 
A basis for $\so(n)$ consists of the matrices $J_{ij}, i< j$  that correspond to the linear maps
\[  J_{ij}\;\colon\; e_i\mapsto e_j,\; e_j\mapsto -e_i,\; e_k\mapsto 0, \;k\ne i,j\]
Identify $\Lambda^1=\RR^n$, $\mathrm{End}(\Lambda^1)=M_n(\RR)$. Then a simple computation of the commutators $[e_ie_j,e_k]$ in $C_n$ shows that
\[ d\rho\;\colon\; \Lambda^2\cong \so(n) \;\colon\; \frac{1}{2}e_ie_j \mapsto J_{ij}\]
Define  the spin group $\Spin$ as the exponential of $\Lambda^2$ in $C_n$,
\[ \Spin :=\{\sum_{k=0}^\infty \frac{1}{k!} \alpha^k \in C_n  \;\colon\; \alpha\in \Lambda^2\} \subset C_n\]
The Lie algebra representation $d\rho\;\colon\;\Lambda^2\to \so(n)$ induces a group representation 
\begin{align*}
 \rho\; &\colon\; \Spin \to \SO\\
\rho(g).v &:= g v g^{-1} \qquad g\in \Spin\subset C_n, \;v\in \Lambda^1=\RR^n
\end{align*}
For $n\ge 2$, the group $\Spin$ is the connected double cover of $\SO$.

$\Spinc$ is the exponential of $\Lambda^2\otimes\CC$ in the complexified Clifford algebra $C_n\otimes \CC$.
  $\Spinc$ is isomorphic to the quotient of $\mathrm{Spin}(n)\times U(1)$
by the 2 element group $\{(1,1), (\varepsilon, -1)\}\cong \ZZ/2\ZZ$,
where $\varepsilon$ is the non-identity element in the kernel of $\Spin\to \SO$,
\[ \mathrm{Spin}^c(n)\cong \mathrm{Spin}(n)\times_{\ZZ/2\ZZ} U(1)\] 
The group $\Spinc$ acts on $\RR^n$ via its surjection onto $\SO$.

\subsection{The spin representation}
For $n=2r$ even, we have an isomorphism
\[ c\;\colon\;C_n\otimes \CC \cong M_{2^r}(\CC)\qquad c(e_j)=E_j\]
where $E_1, \dots, E_n$ are the $2^r\times 2^r$ complex matrices defined in section \ref{Dirac_Rn}.

The vector space $\CC^{2^r}$ is the vector space of a representation of $\Spin$
and $\Spinc$  known  as the spin representation.
The spin representation is constructed via the inclusion of $\Spinc$ in the complexified Clifford algebra $C_n\otimes \CC$.

The representation of $\Spinc$ on $\CC^{2^r}$ gives a representation of $\Spinc$ on $\mathrm{End}(\CC^{2^r})\cong \CC^{2^r}\otimes (\CC^{2^r})^*$,
and the map
\[ c\;\colon\; \RR^n\to \mathrm{End}(\CC^{2^r})\qquad c(\xi)=\sum_{j=1}^n \xi_jE_j\]
is a $\Spinc$ equivariant inclusion of vector spaces.
If $n$ is even, the grading operator $\left[
\begin{smallmatrix}
I&\mathbf{0}\\
\mathbf{0}&-I
\end{smallmatrix}
\right]$
is $\Spinc$ equivariant.

The trivial bundle $\RR^n\times \CC^{2^r}$ is a $\Spinc$ equivariant vector bundle on $\RR^n$,
where $\Spinc$ acts on $\RR^n$ and $\CC^{2^r}$ as above.
The (full) symbol of the Dirac operator is $\sigma(x,\xi)=ic(\xi)$,
and so the action of $D$ on $C_c^\infty(\RR^n, \CC^{2^r})$ is $\Spinc$ equivariant.

An excellent reference for $\Spin$, $\Spinc$ and Clifford algebras is \cite{ABS}.

\subsection{Spin$^c$ vector bundles}

Let $F$ be a  $C^\infty$ $\RR$ vector bundle on a $C^\infty$ manifold $M$ of fiber dimension $n$.
Then $\F(F)$ is the principal $\GL$ bundle on $M$ whose fiber at $p\in M$
is the set of bases $(v_1, v_2, \dots, v_n)$ for the fiber $F_p$.   
\begin{definition}
A Spin$^c$ datum for a $C^\infty$ $\RR$ vector bundle $F$ on a $C^\infty$ manifold $M$
is a pair $(P,\eta)$ where $P$ is a $C^\infty$   principal $\Spinc$ bundle  on $M$ and  $\eta\,\colon P\to \F(F)$ is   a homomorphism of principal bundles that is compatible with the homomorphism $\rho\,\colon \Spinc\to\GL$ in the sense that there is commutativity in the diagram
\[ \xymatrix{   P\times \Spinc \ar[r]\ar[d]_{\eta\times\rho} & P \ar[d]^\eta& \\
 \F(F)\times \GL \ar[r] & \F(F)}
\]
\end{definition}
\begin{definition}
A Spin$^c$ vector bundle is a $\CC^\infty$ $\RR$ vector bundle with a given Spin$^c$ datum.
\end{definition}


An {\em isomorphism} of Spin$^c$ data $(P,\eta), (P',\eta')$ for a vector bundle $F$ is an isomorphism of the principal $\Spinc$ bundles $P\cong P'$ that commutes with the maps $\eta, \eta'$.
Two Spin$^c$ data $(P,\eta), (P',\eta')$ for $F$ are {\em homotopic} if there exists a $C^\infty$ homotopy $(P,\eta_t)$ of Spin$^c$ data such that $\eta_0=\eta$ and $(P,\eta_1)$ is isomorphic to $(P',\eta')$.

A {\em Spin$^c$ structure} for $F$ is an isomorphism class of Spin$^c$ data.
A {\em Spin$^c$ orientation} of $F$ is a homotopy class of  Spin$^c$ data.
\footnote{
A Spin$^c$ orientation determines a $K$-orientation.
However, not every $K$-orientation is Spin$^c$.}

Note that a Spin$^c$ structure is a Euclidean structure plus a Spin$^c$ orientation.
A Spin$^c$ orientation determines an orientation in the usual sense.
$F$ is Spin$^c$ orientable if and only if it is orientable (i.e. the first Stiefel-Whitney class $w_1(E)$ is zero)
and the second Stiefel-Whitney class $w_2(F)$ is in the image of the  map $H^2(M,\ZZ)\to H^2(M,\ZZ/2\ZZ)$.

A Spin$^c$ structure changes the structure group of $F$ from $\GL$ to $\Spinc$, in the sense that it determines an  isomorphism 
\[ P\times_{\Spinc} \RR^n \cong F\]
The spinor bundle $S_F$ of $F$ is the complex vector bundle on $M$
\[S_F=P\times_{\Spinc} \CC^{2^r}\]
$S_F$ comes equipped with a vector bundle map
\[ c\;\colon\;F\to \mathrm{End}(S_F)\]
called the Clifford action of $F$ on $S_F$.
If the fiber dimension of $F$ is even, then $S_F$ is $\ZZ/2\ZZ$ graded, and the Clifford action is odd.

\subsection{The Thom class of a Spin$^c$ vector bundle}\label{Thom}

If $M$ is compact and the fiber dimension of $F$ is even, the Spin$^c$ datum for $F$ determines an element 
\[\lambda_F=(\pi^*S_F^+, \pi^*S_F^-,c)\in K^0(F)\]
as follows.
Denote the projection of $F$ onto $M$ by $\pi:F\to M$.
The pull-back of $S_F$ to $F$ is a direct sum $\pi^*S_F=\pi^*S_F^+\oplus\pi^*S_F^-$,
and the Clifford action of $F$ is a vector bundle map $\pi^*S_F^+\to \pi^*S_F^-$
which is an isomorphism outside the zero section of $F$.

The dual (or conjugate) of $\lambda_F$ will be denoted by $\tau_F$,
and will be referred to as the {\em Thom class} of the Spin$^c$ vector bundle $F$.
Thus, $\tau_F$ is
\[\tau_F=(\pi^*\overline{S_F^+}, \pi^*\overline{S_F^-},c)\in K^0(F)\]
where $\overline{S_F^+},\; \overline{S_F^-}$ are the conjugate vector bundles of $S_F^+,\; S_F^-$.
Note that $c$ is unchanged.

The Thom class $\tau_F$ restricts in each fiber $F_p, p\in M$ to the Bott generator element of the oriented vector space $F_p$, i.e. $\mathrm{ch}(\tau_F|F_p)[F_p]=1$.

\subsection{Spin$^c$ manifolds}\label{spincm}

\begin{definition}
A Spin$^c$ manifold is a $C^\infty$ manifold $M$ (with or without boundary) whose tangent bundle $TM$ is a Spin$^c$ vector bundle.
\end{definition}
Every Spin manifold is Spin$^c$.
Also, every stably almost complex manifold is Spin$^c$ oriented.
This includes complex manifolds, symplectic manifolds, and contact manifolds.
Most of the oriented manifolds that occur in practice are Spin$^c$ oriented.

Let $\Omega$ be a Spin$^c$ manifold with boundary $\partial \Omega=M$.
The Spin$^c$ datum for $\Omega$ determines a Spin$^c$ datum for $M$ as follows.

The frame bundle $\F(TM)$ injects into $\F(T\Omega)$ by adding the outward unit normal vector ${\bf n}$ as the first vector to each frame,
\[ \F(TM)\to \F(T\Omega)\qquad (v_1, v_2, \dots,v_n)\mapsto ({\bf n}, v_1, v_2, \dots, v_n)\] 
The principal $\Spinc$ bundle $P^M$ of $M$ is the preimage of the principal $\mathrm{Spin}^c(n+1)$ bundle $P^\Omega$ of $\Omega$ under this map,
\[ \xymatrix{   P^M \ar[r]\ar[d] & P^\Omega \ar[d]& \\
 \F(TM)\ar[r] & \F(T\Omega)}
\] 
If $\Omega$ is odd dimensional, then the spinor bundle of $M$ is the restriction to $M$ of the spinor bundle of $\Omega$.
The grading operator is the Clifford action $ic({\bf n})$, where ${\bf n}$ is the outward pointing unit normal vector.

\subsection{The 2-out-of-3 lemma}

The construction of the Spin$^c$ orientation of a boundary is a special case of the 2-out-of-3 principle.

There is a canonical homomorphism $\mathrm{Spin}^c(k)\times \mathrm{Spin}^c(l)\to \mathrm{Spin}^c(k+l)$
that gives commutativity in the diagram
\[ \xymatrix{  \mathrm{Spin}^c(k)\times \mathrm{Spin}^c(l)  \ar[r]\ar[d] &\mathrm{Spin}^c(k+l) \ar[d]& \\
 \mathrm{SO}(k)\times \mathrm{SO}(l)\ar[r] & \mathrm{SO}(k+l)}
\]
Hence, if  $F_1, F_2$ are two Spin$^c$ vector bundles on $M$,
then there is a straightforward Spin$^c$ datum for the direct sum $F_1\oplus F_2$,
and therefore the direct sum of two Spin$^c$ oriented vector bundles is Spin$^c$ oriented.

The construction is such that the Spin$^c$ line bundle of $F_1\oplus F_2$ is the tensor product of the Spin$^c$ line bundles of $F_1$ and $F_2$.
As a result, the Todd class is multiplicative for Spin$^c$ vector bundles,
\[ \mathrm{Td}(F_1\oplus F_2) = \mathrm{Td}(F_1)\cup \mathrm{Td}(F_2)\]

\begin{lemma}\label{2outof3}
Let $F_1, F_2$ be two $C^\infty$ $\RR$ vector bundles on $M$.
Assume given Spin$^c$ orientations for $F_1$ and for $F_1\oplus F_2$.
Then there exists a unique Spin$^c$ orientation for $F_2$ such that the direct sum Spin$^c$ orientation of $F_1\oplus F_2$ is the given one.
\end{lemma}
In summary, given a short exact sequence of $\RR$ vector bundles
\[ 0\to F_1\to F_3\to F_2\to 0\]
where 2 out of 3 are Spin$^c$ oriented, a Spin$^c$ orientation is then determined for the third.

\subsection{The Dirac operator of a Spin$^c$ manifold}

On a Spin$^c$ manifold $M$ there is a first order elliptic differential operator known as its Dirac operator.
The symbol of the Dirac operator on $\RR^n$ is the $\Spinc$ equivariant map,
\[ ic\;\colon\;\RR^n\to \mathrm{End}(\CC^{2^r})\]
Via the given isomorphism $TM\cong P\times_{\Spinc} \RR^n$ we then obtain the principal symbol of an operator on $M$,
\[ \sigma(p)\;\colon\; T^*_pM\to \mathrm{End}(S_p)\qquad p\in M\] 
where $S$ is the spinor bundle of the Spin$^c$ vector bundle  $TM$.
The Dirac operator of $M$ is a first order differential operator $D$ on $M$ whose principal symbol is this symbol.
Thus $D$ is a linear map
\[ D\;\colon\; C_c^\infty(M,S)\to C_c^\infty(M,S)\]
$D$ is unique up to lower order terms.

If $n$ is even, the $\Spinc$ equivariant grading operator 
$\left[
\begin{smallmatrix}
I&\mathbf{0}\\
\mathbf{0}&-I
\end{smallmatrix}
\right]$
of $\CC^{2^r}$ determines a grading operator for the spinor bundle $S$ of any $n$-dimensional Spin$^c$ manifold $M$.
For an appropriate choice of lower order terms the Dirac operator $D$ of $M$
anticommutes with the grading.

\subsection{Twisting by vector bundles}

Let $M$ be a Spin$^c$ manifold with Dirac operator $D$.
Suppose that $E$ is a $C^\infty$ $\CC$ vector bundle on $M$.
We can twist the symbol of the Dirac operator by $E$,
\[ \sigma_E(p)\;\colon\; T^*_pM\to \mathrm{End}(S_p\otimes E_p)\qquad \sigma_E(p)=\sigma(p)\otimes I_{E_p}\]
The Dirac operator twisted by $E$  is a first order differential operator $D_E$ on $M$ whose principal symbol is $\sigma_E$,
\[D_E \;\colon\; C_c^\infty(M,S\otimes E) \to C_c^\infty(M,S\otimes E)\]
$D_E$ is determined up to lower order terms.
In the even case the bundle $S\otimes E$ is graded and (for an appropriate choice of lower order terms) $D_E$ anticommutes with the grading. 
Hence $D_E$ gives a map
\[D_E^+ \;\colon\; C_c^\infty(M,S^+\otimes E) \to C_c^\infty(M,S^-\otimes E)\]
If $M$ is closed, this map $D_E^+$ has finite dimensional kernel and cokernel, and the index of $D_E$ is the difference of these two dimensions
\[ \Ind \,D_E := \mathrm{dim \,Kernel} \,D_E^+ - \mathrm{dim \,Cokernel} \,D_E^+\]
The index of an elliptic operator depends only on its highest order part,
and therefore $\Ind\, D_E$ is well defined.

\subsection{The Bott generator vector bundle}\label{Bottvb}

For $n=2r+1$
restrict the map $c\;\colon \RR^{n}\to \mathrm{End}(\CC^{2^r})$ to the unit sphere $S^{2r}\subset \RR^n$,
\[ c\;\colon \; S^{2r}\to \mathrm{End}(\CC^{2^r})\qquad c(t)=\sum_{j=1}^{2r+1} t_jE_j\]
If $\sum t_j^2=1$ we obtain
\[ c(t)^2=\left(\sum_{j=1}^n t_jE_j\right)^2=\sum_{j=1}^n t_j^2E_j^2+\sum_{j< k} t_jt_k(E_jE_k+E_kE_j) = -I_{2^r}\]
$ic(t)$ is the grading operator for the spinor bundle $S^{2r}\times \CC^{2^r}$ of $S^{2r}$, where $S^{2r}$ has the Spin$^c$ structure it receives as the boundary of the unit ball in $\RR^{2r+1}$.

The Bott generator vector bundle $\beta$ is the positive spinor bundle if $r$ is even,
and the negative spinor bundle if $r$ is odd. 
Equivalently, in all cases $\beta$ is the dual of the positive spinors.
The Bott generator vector bundle is determined, up to isomorphism, by the following two properties
\begin{itemize}
\item The fiber dimension of $\beta$ is $2^{r-1}$.
\item $\mathrm{ch}(\beta)[S^{2r}]=1$
\end{itemize}
The Bott generator vector bundle $\beta$ is $\mathrm{Spin}^c(n+1)$ equivariant, where $\mathrm{Spin}^c(n+1)$ acts on $S^n\subset \RR^{n+1}$ via its isometry group $\mathrm{SO}(n+1)$.


\begin{proposition}\label{chbeta}
If $\beta$ is the Bott generator vector bundle on $S^{2r}$ then
\[\mathrm{ch}(\beta)[S^{2r}]=1\]
\end{proposition}
\begin{proof}
The positive spinor bundle $S^+$ corresponds to the projection valued function
\[ e\;\colon\; S^{2r}\to \mathrm{End}(\CC^{2^r})\qquad e=\frac{1}{2}(1+i\sum_{j=1}^{2r+1} t_jE_j)\]
The Chern character of $S^+$ is represented by a differential form
whose component in dimension $2r$ is
\[ \frac{i^r}{r!\,(2\pi)^r}  \mathrm{trace}( e\,(de)^{2r})\] 
We obtain
\[(de)^{2r}=\frac{i^{2r}(2r)!}{2^{2r}} \sum dt_1\cdots \widehat{dt_j}\cdots dt_{2r+1} \cdot E_1\cdots \widehat{E_j}\cdots E_{2r+1}\]
From $i^{r+1}E_1E_2\cdots E_{2r+1}=I$ we get $i^{r+1}E_1\cdots \widehat{E_j}\cdots E_{2r+1}=(-1)^jE_j$ and so
\[(de)^{2r}=\frac{i^{r-1}(2r)!}{2^{2r}} \sum (-1)^j\, dt_1\cdots \widehat{dt_j}\cdots dt_{2r+1} \cdot E_j\]
The trace of the matrices $E_j$ and $E_jE_k$ with $j\ne k$ is zero, while the trace of $E_j^2=-I$ is $-2^r$.
Thus
\[  \frac{i^r}{r!\,(2\pi)^r}  \mathrm{trace}( e\,(de)^{2r}) = \frac{i^{2r}(2r)!}{(2\pi)^r 2^{r+1}(r!)} \sum (-1)^{j-1}\, t_j dt_1\cdots \widehat{dt_j}\cdots dt_{2r+1} \]
By Stokes's Theorem,
\[ \int_{S^{2r}} \mathrm{ch}(S^+) =  \frac{i^{2r}(2r)!}{\pi^r 2^{2r+1}(r!)} \int_{B^{2r+1}} (2r+1)\,dt_1\cdots dt_{2r+1} = (-1)^r\]

\end{proof}


\subsection{An index 1 operator}\label{indexone}

For the next proposition, recall that if $V$ is a finite dimensional vector space, then there is a canonical nonzero element in $V\otimes V^*$,
which maps to the identity map under the isomorphism $V\otimes V^*\cong \mathrm{Hom}(V,V)$. 

\begin{proposition}\label{index1}
If $D$ is the Dirac operator of the even dimensional sphere $S^n$
with the Spin$^c$ datum it receives as the boundary of the unit ball in $\RR^{n+1}$,
then
\[ \mathrm{Index} \;D_\beta = 1\]
More precisely, $D_\beta$ has zero cokernel, and the kernel of $D_\beta$ is spanned by the canonical section of $S^+\otimes \beta=S^+\otimes (S^+)^*$.
\end{proposition}
\begin{proof}
On any even dimensional Spin$^c$ manifold $M$, there is a canonical isomorphism of vector bundles
\[ S\otimes S^* \cong \Lambda_\CC TM\]
where $S$ is the spinor bundle of $M$.
This is implied by the fact that, as representations of $\Spinc$,
\[ \CC^{2^r}\otimes (\CC^{2^r})^*\cong \Lambda_\CC \RR^n\]
Via this isomorphism,  $D_{S^*}$ identifies (up to lower order terms) with $d+d^*$,
where $d$ is the de Rham operator, and $d^*$ its formal adjoint.
Note that the kernel of $D_{S^*}$ is the same as the kernel of $D_{S^*}^2=(d+d^*)^2$,
i.e. it consists of harmonic forms.
On $S^n$ the only harmonic forms are the constant functions, and scalar multiples of the standard volume form.

Because the Bott generator vector bundle $\beta$ is dual to $S^+$,
$S^+\otimes \beta\cong \mathrm{Hom}(S^+,S^+)$ contains a trivial line bundle,
which identifies with the line bundle in $\Lambda^0\oplus \Lambda^n$
spanned by $(1,\omega)$, where $\omega$ is the standard volume form.
This follows from  representation theory. 
Thus, the kernel of $D_\beta$ is the one dimensional vector space spanned by the harmonic forms $c+c\omega$, $c\in \CC$.
The intersection of $\Lambda^0\oplus \Lambda^n$ with $S^-\otimes \beta$ is zero,
and so the cokernel of $D_\beta$ is zero.

\end{proof}

\begin{remark}
Propositions \ref{chbeta} and \ref{index1} verify, by direct calculation, a special case of the index formula for an operator of index 1,
\[ \Ind\,D_\beta = \mathrm{ch}(\beta)[S^{2r}]\]
\end{remark}

\section{The Group $K_0(\cdot)$}

Define an abelian group denoted $\Kpt$ by considering pairs $(M,E)$ such that $M$ is a compact even-dimensional Spin$^c$ manifold without boundary, and $E$ is a $C^\infty$ $\CC$ vector bundle on $M$.

Define $\Kpt=\{(M,E)\}/\sim$ where the equivalence relation $\sim$ is generated by the three elementary steps
\begin{itemize}
\item Bordism
\item Direct sum - disjoint union
\item Vector bundle modification 
\end{itemize}
Addition in $\Kpt$ is disjoint union
\[ (M,E)+(M',E') = (M\sqcup M', E\sqcup E')\]

\subsection{Additive inverse}

In $\Kpt$ the additive inverse of $(M,E)$ is $(-M,E)$ where $-M$ denotes $M$ with the Spin$^c$ structure reversed,
\[ -(M,E) = (-M,E)\] 
To reverse the Spin$^c$ structure on a Spin$^c$ manifold $M$ of dimension $n$,
the Spin$^c$ datum of $M$ is modified as follows.
The Riemannian metric is unchanged. The orientation is reversed. 
The new $\SO$ principal bundle consists of those orthonormal frames that were previously negatively oriented,
\[ \FSO^-(TM)=\FSO(TM)\times_{\SO} O^-(n)\]
Here $O^-(n)$ is the set of orthogonal $n\times n$ matrices with determinant $-1$
with $\SO$ acting on $O^-(n)$ from the left and from the right by matrix multiplication.
Denote by $\tilde{O}_-(n)$ the connected double cover of $O^-(n)$,
and observe that the actions of $\SO$ on $O^-(n)$ lift to give left and right actions of $\Spinc$ on $\tilde{O}^-(n)$.
The principal $\Spinc$ bundle $P$ is replaced by 
\[P^- = P\times_{\Spinc} \tilde{O}^-(n)\]
with the evident map to $\FSO^-(TM)=\FSO(TM)\times_{\SO} O^-(n)$.

The spinor bundle of $-M$ is the same as the spinor bundle of $M$.
When $M$ is even dimensional the Clifford action is unchanged, and the grading is reversed.
In the odd dimensional case, the Clifford action $c\;\colon TM\to \mathrm{End}(S)$ is replaced by $-c$.

\subsection{Bordism}

$(M,E)$ is isomorphic to $(M',E')$ if there exists a diffeomorphism 
\[ \psi\;\colon\; M\to M'\]
that preserves the Spin$^c$ structures, in the sense that the pullback $(\psi^*P',\psi^*\eta')$ of the Spin$^c$ datum $(P',\eta')$ of $M'$ is isomorphic to the Spin$^c$ datum $(P,\eta)$ of $M$,
and the pullback $\psi^*E'$ is isomorphic (as a $\CC$ vector bundle on $M$) to $E$,
\[ \psi^*(E')\cong E\]

\begin{definition}
$(M_0, E_0)$ is {\em bordant} to $(M_1, E_1)$ if there exists a pair $(\Omega, E)$ such that
\begin{itemize}
\item $\Omega$ is a compact odd-dimensional Spin$^c$ manifold with boundary.
\item $E$ is a $C^\infty$ $\CC$ vector bundle on $\Omega$.
\item $(\partial \Omega, E|_{\partial\Omega}) \cong (M_0, E_0)\sqcup (-M_1, E_1)$
\end{itemize}
Here the boundary $\partial\Omega$ receives the Spin$^c$ datum as above, and $-M_1$ is $M_1$ with the Spin$^c$ structure reversed.
\end{definition}

\[
\begin{tikzpicture}[line width=1pt,scale=.6]
\begin{scope}
\begin{scope}
\clip (0,-3) rectangle (1.5*1.414,3);
\draw (0,0) circle (3cm);
\draw (0,0) circle (1cm);
\end{scope}
\begin{scope}
\clip (1.5*1.414,-3) rectangle (1.5*1.414+2.19,2.5);
\draw (1.5*1.414+2.19,4.18) circle (3cm); 
\draw (1.5*1.414+2.19,-4.18) circle (3cm); 
\end{scope}
\end{scope}
\begin{scope}[xscale=-1,xshift=-8.62cm]
\begin{scope}
\clip (0,-3) rectangle (1.5*1.414,3);
\draw (0,0) circle (3cm);
\draw (0,0) circle (1cm);
\end{scope}
\begin{scope}
\clip (1.5*1.414,-3) rectangle (1.5*1.414+2.19,2.5);
\draw (1.5*1.414+2.19,4.18) circle (3cm); 
\draw (1.5*1.414+2.19,-4.18) circle (3cm); 
\end{scope}
\end{scope}
\begin{scope}[xshift=8.62cm]
\begin{scope}
\clip (0,-3) rectangle (1.5*1.414,3);
\draw (0,0) circle (3cm);
\draw (0,0) circle (1cm);
\end{scope}
\begin{scope}
\clip (1.5*1.414,-3) rectangle (1.5*1.414+2.19,2.5);
\draw (1.5*1.414+2.19,4.18) circle (3cm); 
\draw (1.5*1.414+2.19,-4.18) circle (3cm); 
\end{scope}
\end{scope}
\draw (0,2) ellipse (.2cm and 1cm);
\draw (0,-2) ellipse (.2cm and 1cm);
\begin{scope}
\clip (12.95,-2) rectangle (14,2);
\draw (12.95,0) ellipse (.25cm and 1.2cm);
\end{scope}
\begin{scope}
\clip (12.95,-2) rectangle (12,2);
\draw[dashed] (12.95,0) ellipse (.25cm and 1.2cm);
\end{scope}
\begin{scope}
\clip (8.6,-3) rectangle (9,3);
\draw (8.6,2) ellipse (.2cm and 1cm);
\draw (8.6,-2) ellipse (.2cm and 1cm);
\end{scope}
\begin{scope}
\clip (8.6,-3) rectangle (8,3);
\draw[dashed] (8.6,2) ellipse (.2cm and 1cm);
\draw[dashed] (8.6,-2) ellipse (.2cm and 1cm);
\end{scope}
\draw (1,-4) node {$(M_0,E_0$)};
\draw (12,-4) node {$(-M_1,E_1$)};
\end{tikzpicture}
\]
The bordism class of a pair  $(M,E)$ only depends on the Spin$^c$ oriented manifold $M$ and the vector bundle $E$.  

\subsection{Direct sum - disjoint union}

Let $E, E'$ be two $\CC$ vector bundles on $M$, then 
\[ (M,E)\sqcup (M,E')\sim (M, E\oplus E')\]

\subsection{Vector bundle modification}\label{vbm}

Let $F$ be a Spin$^c$ vector bundle on $M$ with even fiber dimension $n=2r$.
Part of the Spin$^c$ datum of $F$ is a principal $\Spinc$ bundle $P$ on $M$,
\[F=P\times_\Spinc \RR^n\]
The Bott generator vector bundle $\beta$ on $S^n$ is $\Spinc$ equivariant.
Therefore, associated to $P$ we have a fiber bundle $\pi\,\colon \Sigma F \to M$ whose fibers are oriented spheres of dimension $n$,
\[ \Sigma F = P\times_\Spinc S^n\]
and  a vector bundle on $\Sigma F$
\[ \beta_F = P\times_\Spinc \beta\]
If $M$ is a Spin$^c$ manifold,  the total space of $F$ is Spin$^c$, because (not canonically) $TF = \pi^*F\oplus \pi^*TM$, 
and the direct sum of two Spin$^c$ vector bundles is Spin$^c$.
Every trivial bundle is Spin$^c$, so the total space of $F\oplus \underline{\RR}$ is Spin$^c$.
$\Sigma F$ is a Spin$^c$ manifold as the boundary of the unit ball bundle of $F\oplus \underline{\RR}$.

Then
\[ (M,E)\sim (\Sigma F, \beta_F\otimes \pi^*E)\]

\subsection{The equivalence relation}

The equivalence relation $\sim$ on the collection $\{(M,E)\}$ is the equivalence relation generated by the above  elementary steps --- i.e. $(M,E)\sim (M',E')$ if it is possible to pass from $(M,E)$ to $(M',E')$ by a finite sequence of the three elementary steps.

$(M,E)=0$ in $K_0(\cdot)$ if and only if $(M,E)\sim (M',E')$ where $(M',E')$ bounds.

\section{Proof of the index theorem}

\subsection{Sphere lemma}

\begin{lemma}\label{sphere}
Let $(M,E)$ be a pair as above.
Then there exists a positive integer $r$ and a $\CC$ vector bundle $F$ on $S^{2r}$
such that $(M,E)\sim (S^{2r},F)$. 
Here $S^{2r}$ has the Spin$^c$ datum it receives as the boundary of the unit ball in $\RR^{2r+1}$.
\end{lemma}
\begin{proof}
Embed $M$ in $\RR^{2r}$.
By the 2-out-of-3 principle the normal bundle $\nu$
\[ 0\to TM\to M\times \RR^{2r}\to \nu\to 0\]
is Spin$^c$ oriented.
\vskip 6pt

Step 1. $(M,E)\sim(\Sigma \nu, \beta_\nu\otimes \pi^*E)$. (Vector bundle modification by $\nu$.)

\vskip 6pt

$\Sigma \nu$ bounds, therefore it is bordant to $S^{2r}$.
We shall use a specific bordism from $\Sigma \nu$ to $S^{2r}$.
Since $M$ is compact we may assume that the embedding of $M$ in $\RR^{2r}$
embeds $M$ in the interior of the unit ball of $\RR^{2r}$. 
Using the inclusion $\RR^{2r}\to \RR^{2r+1}$, 
a compact tubular neighborhood of $M$ in $\RR^{2r+1}$ 
identifies with the ball bundle $B(\nu\oplus \underline{\RR})$, whose boundary is $\Sigma \nu$.
Let $\Omega$ be the unit ball of $\RR^{2r+1}$ with the interior of $B(\nu\oplus \underline{\RR})$ removed.
Then $\Omega$ is a bordism of Spin$^c$ manifolds from $\Sigma \nu$ to $S^{2r}$.

By Lemma \ref{L} below, with $B=B(\nu\oplus \underline{\RR})$ and $S=S(\nu\oplus \underline{\RR})$ there exists a $\CC$ vector bundle $L$ on $B$
such that $(\beta_\nu\otimes \pi^*E)\oplus L|_S$ extends to  a vector bundle $F$ on $\Omega$.
Note that $(\Sigma \nu, L|_S)$ is the boundary of $(B,L)$.
\vskip 6pt

Step 2. $(\Sigma \nu, \beta_\nu\otimes \pi^*E)\sim (\Sigma \nu, \beta_\nu\otimes \pi^*E)\sqcup (\Sigma\nu, L|_S)$. (Bordism.)

\vskip 6pt

Step 3. $(\Sigma \nu, \beta_\nu\otimes \pi^*E)\sqcup (\Sigma\nu, L|_S)\sim (\Sigma \nu, (\beta_\nu\otimes \pi^*E)\oplus \pi^*L)$. (Direct sum - disjoint union.)

\vskip 6pt

Step 4. $(\Sigma \nu, (\beta_\nu\otimes \pi^*E)\oplus \pi^*L)\sim (S^{2r}, F|_{S^{2r}})$. (Bordism.)

\end{proof}

\begin{lemma} \label{L}
Let the unit ball in $\RR^n$ be the union of two compact sets $B, \Omega$ and let $S=B\cap \Omega$.
Given a $\CC$ vector bundle $E$ on $S$
there exists a $\CC$ vector bundle $L$ on $B$ such that $E\oplus L|_S$ extends to $\Omega$.  
\end{lemma}
\begin{proof}
Exactness of the $K$-theory Mayer-Vietoris sequence
\[ \ldots\to K^0(\Omega)\oplus K^0(B)\to K^0(S)\to K^1(B\cup \Omega)=0\to \ldots \]
shows that the lemma is true for $K$-theory classes,
i.e., there exists $[L]-[\underline{\CC}^k]\in K^0(B)$ and $[F]-[\underline{\CC}^m]\in K^0(\Omega)$
such that in $K^0(S)$
\[ [E]+[L|_S]-[\underline{\CC}^k]=[F|_S]-[\underline{\CC}^m]\]
This means that the vector bundle $E\oplus L|_S$ is stably isomorphic to $F|_S$.
By adding trivial bundles to $F$ and $L$, if needed, we obtain the result.

\end{proof}

\subsection{Bott periodicity}

Bott periodicity is the following statement about the homotopy groups of $\mathrm{GL}(n,\CC)$,
\[ \pi_j\,\mathrm{GL}(n,\CC)=\left\{
\begin{array}{ll}  \ZZ & j\;\mbox{odd}\\ 0& j\;\mbox{even} \end{array}
\right.\]
\[j=0,1,2,\dots, 2n-1\]
\begin{proposition}\label{Bott}
Let $E$ be a $\CC$ vector bundle on $S^{2r}$,
then there exist non-negative integers $l,m$ and an integer $q$ such that
\[ E\oplus \underline{\CC}^l \cong q\beta\oplus \underline{\CC}^m\]
If $q>0$ then $q\beta=\beta\oplus\cdots \oplus\beta$, and if $q<0$ then $q\beta=|q|\beta^*$. 
\end{proposition}
\begin{proof}
Up to isomorphism, $E$ is determined by the homotopy type of a map $S^{2r-1}\to \mathrm{GL}(k,\CC)$, i.e. an element of $\pi_{2r-1} \mathrm{GL}(k,\CC)$. 
Here $S^{2r-1}$ is the equatorial sphere in $S^{2r}$,
and $k$ is the fiber dimension of $E$.
By Bott periodicity, if $k$ is sufficiently large, $\pi_{2r-1} \mathrm{GL}(k,\CC)\cong \ZZ$.
The Bott generator vector bundle $\beta$ corresponds to the generator of $\pi_{2r-1} \mathrm{GL}(2^{r-1},\CC)$.

\end{proof}

\begin{corollary}\label{cor}
Given a pair  $(M,E)$ 
there exists a positive integer $r$ and an integer $q$ 
such that $(M,E)\sim (S^{2r},q\beta)$. 
Here $S^{2r}$ has the Spin$^c$ datum it receives as the boundary of the unit ball in $\RR^{2r+1}$.
\end{corollary}
\begin{proof}
For any integer $l$, $(S^{2r}, \underline{\CC}^l)$ bounds.
The corollary now follows by combining the sphere lemma \ref{sphere} with Proposition \ref{Bott}.
\end{proof}

\subsection{The index theorem}

We define two homomorphisms of abelian groups
\[ K_0(\cdot)\to \ZZ\qquad (M,E)\mapsto \Ind\, (D_E)\]
\[ K_0(\cdot)\to \QQ\qquad (M,E)\mapsto (\mathrm{ch}(E)\cup \mathrm{Td}(M))[M]\]
We shall prove  in sections \ref{aind} and \ref{tind} below that these are well-defined homomorphisms, i.e. that each map is compatible with the three elementary moves.

\begin{proposition}\label{isoZ}
The map 
\[K_0(\cdot)\to \ZZ\qquad (M,E)\mapsto \Ind\, (D_E)\]
is an isomorphism of abelian groups.
\end{proposition}
\begin{proof}
Injectivity is a corollary of Bott periodicity. 
To prove injectivity, assume that $\Ind\,(D_E)=0$  for a pair  $(M,E)$.
By Corollary \ref{cor}, 
\[(M,E)\sim (S^{2r}, q\beta)\]
and therefore $\Ind\, D_{q\beta}=0$.
By Proposition \ref{index1} this implies that $q=0$,
i.e. $q\beta$ is the zero vector bundle.
Since $(S^{2r}, 0)$ evidently bounds, $(M,E)\sim 0$.

Bott periodicity is not used in the proof of surjectivity.
Surjectivity follows from the existence of a pair $(M,E)$ with $\Ind\,(D_E)=1$.
For example, $(M,E)=(S^{2r},\beta)$ or $(\CC P^n, \underline{\CC})$.
\end{proof}

\begin{remark}
Proposition \ref{isoZ} says that $\Ind\,D_E$ is a complete invariant for the equivalence relation generated by the three elementary steps,
i.e. $(M,E)\sim (M',E')$ if and only if $\Ind\, (D_E)=\Ind\, (D'_{E'})$.
\end{remark}

We now prove the index Theorem \ref{thm}.

\begin{proof}
Consider the two homomorphism of abelian groups 
\[ K_0(\cdot)\to \ZZ\qquad (M,E)\mapsto \Ind\, (D_E)\]
\[ K_0(\cdot)\to \QQ\qquad (M,E)\mapsto (\mathrm{ch}(E)\cup \mathrm{Td}(M))[M]\]
Since the first one is an isomorphism, 
to show that these two are equal it will suffice to show that they are equal for one example of index 1.

The  Spin$^c$ structure of $S^{2r}$ that it receives as the boundary of the unit ball in $\RR^{2r+1}$
is, in fact, a Spin structure.
Since the tangent bundle $TS^{2r}$ is stably trivial, 
\[\mathrm{Td}(S^{2r})=\hat{A}(S^{2r})=1\]
The Bott generator vector bundle $\beta$ on $S^{2r}$ has the property that
$\mathrm{ch}(\beta)[S^{2r}]=1$,
while also $\Ind\,D_\beta = 1$. This completes the proof.
\end{proof}

\section{The analytic index}\label{aind}

In this section we prove that the homomorphism of abelian groups 
\[ K_0(\cdot)\to \ZZ\qquad (M,E)\mapsto \Ind\, (D_E)\]
is well-defined, i.e. that the three elementary moves are index preserving.
For direct sum - disjoint union this is immediate.

\subsection{Bordism invariance}

There are various proofs of bordism invariance of the analytic index \cite{Pa63}. 
Here we outline a proof of Nigel Higson \cite{Hi91}.

Let $M$ be a compact even dimensional Spin$^c$ manifold that is the boundary of a complete Spin$^c$ manifold $W$, $\partial W=M$.
$W$ is not necessarily compact.
Let $E$ be a smooth $\CC$ vector bundle on $W$.
We may assume that $W$ is the product Spin$^c$ manifold $M\times (-1,0]$ in a neighborhood of the boundary $M$.
Let $W^+$ be the (non-compact and complete) Spin$^c$ manifold $W\cup M\times (0,\infty)$.
$E$ extends in the evident way to $W+$.
\begin{figure}
\includegraphics[width=120mm]{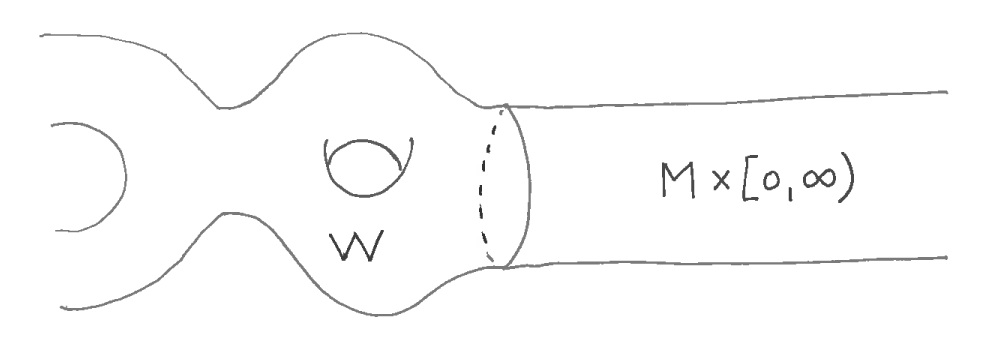}
\end{figure}

Let $D$ be the Dirac operator of $W^+$, and $D_E$ the twisted operator.
Note that $W^+$ is odd-dimensional, so the spinor bundle $S$ of $W+$ is not $\ZZ_2$-graded.

Due to the completeness of $W^+$, $D_E$ is essentially self-adjoint, and the functional calculus applies.
The crucial analytic fact about $D_E$ is that for any $\psi\in C_c^\infty(W^+)$ the operator $\psi(D_E\pm i)^{-1}$ is compact.
This implies that if $\phi\in C^\infty(W^+)$ is locally constant outside of a compact set, then the commutator $[(D_E\pm i)^{-1}, \phi]$ is a compact operator.

Now choose a function $\phi\in C^\infty(W^+)$ such that
\begin{itemize}
\item $\phi=0$ if restricted to  $W$
\item $\phi=1$ if restricted to  $M\times [1,\infty)$.
\end{itemize}
Consider the bounded operator
\[ F_W = I-2i\phi(D_E+i)^{-1}\phi\]
on the Hilbert space $L^2(S|M\times [0,\infty))$.
$F_W$ is a Fredholm operator with parametrix $I+2i\phi(D_E-i)^{-1}\phi$.

If $W$ is replaced by a different complete manifold $V$ with the same boundary $\partial V=M$,
then $F_V$ is a compact perturbation of $F_W$.
If $D_1$ and $D_2$ are the twisted Dirac operators of $W^+$ and $V^+$, then
\begin{align*}
\frac{1}{2i}(F_V-F_W)&= \phi(D_1+i)^{-1}\phi-\phi(D_2+i)^{-1}\phi\\
&=\phi(D_2+i)^{-1}\;((D_2+i)\phi-\phi(D_1+i))\;(D_1+i)^{-1}\phi\\
&=\phi(D_2+i)^{-1}\;(D_2\phi-\phi D_1)\;(D_1+i)^{-1}\phi
\end{align*}
Note that $D_2\phi-\phi D_1$ is a local operator on $M\times [0,\infty)$,
where $D_2=D_1$.
Since $\phi$ is locally constant outside a compact set, the commutator $D_2\phi-\phi D_1$ is Clifford multiplication with a compactly supported section. Therefore $(D_2\phi-\phi D_1)(D_1+i)^{-1}\phi$ is a compact operator.

It follows that $\mathrm{Index}\,F_W=\mathrm{Index}\,F_V$ only depends on $M$.
An elementary calculation shows that if $V=M\times (-\infty,0]$, so that $V^+=M\times \RR$, then $\mathrm{Index}\, F_V$
equals the index of the Dirac operator $D_M$ of $M$ twisted by $E|M$.

Finally, if $W$ is compact, then $\mathrm{Index}\, F=0$.
To see this, consider the direct sum $F\oplus I$ on $L^2(S)=L^2(S|M\times [0,\infty))\oplus L^2(S|W)$,
which is  $I-2i\phi(D_E+i)^{-1}\phi$ (now interpreted as an operator on $L^2(S)$).
If $W$ is compact, then $1-\phi$ is compactly supported,
and  $F\oplus I$ is a compact perturbation of the operator $I-2i(D_E+i)^{-1}=(D_E-i)(D_E+i)^{-1}$ on $W^+$,
which is a unitary (the Cayley transform of $D_E$).

In summary, if $W$ is compact and $M=\partial W$,
then 
\[ \mathrm{Index}\, D_M\otimes (E|M)=\mathrm{Index}\,F=0\]
For details see \cite{Hi91}.

\subsection{Vector bundle modification}

\begin{proposition} Let $(M,E)\sim (\Sigma F, \beta_F\otimes \pi^*E)$ be as in section \ref{vbm}. Then
\[ \mathrm{Index}\, D_{\beta_F\otimes \pi^*E}=\mathrm{Index}\, D_E\]
More precisely, the kernel and cokernel of $D_{\beta_F\otimes \pi^*E}$ have the same dimensions as the kernel and cokernel of $D_E$.
\end{proposition}
\begin{proof}
For each point $p\in M$, the fiber $S(F_p\oplus \RR)$ of $\Sigma F$ at $p$ comes equipped with its Dirac operator $D_p$.
Twist $D_p$ by $\beta_F$ restricted to $S(F_p\oplus \RR)$.
Denote the resulting family of elliptic operators by $\mathfrak{D}_\beta$.

Let $M_1$ and $M_2$ be two even dimensional Spin$^c$ manifolds with Dirac operators $D_1$, $D_2$ and spinor bundles $S_1, S_2$.
The Dirac operator of $M_1\times M_2$ is obtained by a standard construction from $D_1$ and $D_2$,
\[ D_{M_1\times M_2}=D_1\# D_2 =
\left(\begin{array}{cc} 
D_1\otimes 1 & -1\otimes D_2^*\\
1\otimes D_2 & D_1^*\otimes 1
\end{array}\right)
\]
where the $2\times 2$ matrix is an operator from $(S_1^+\otimes S_2^+)\oplus (S_1^-\otimes S_2^-)$ to 
$(S_1^-\otimes S_2^+)\oplus (S_1^+\otimes S_2^-)$.
In particular, if $D_2$ has one dimensional kernel and zero cokernel,
then the kernel and cokernel of $D_{M_1\times M_2}$ identify with the kernel and cokernel of $D_1$.

Due to local triviality of the fibration $\Sigma F\to M$, the Dirac operator of ${\Sigma F}$ is obtained by combining, as above, the Dirac operator of $M$ with the family of elliptic operators $\mathfrak{D}$, and therefore
\[ D_{\beta_F\otimes \pi^*E} = D_E \# \mathfrak{D}_\beta\]
Note that for each $p\in M$ $(\mathfrak{D}_\beta)_p$ identifies with $D_\beta$ as in Proposition \ref{index1}.
Choosing a local trivialization of the fibration $\Sigma F\to M$,
locally the kernel of $D_{\beta_F\otimes \pi^*E}$ identifies with the kernel of $D_E$,
and this identification is independent of the choice of local trivialization 
because the kernel of $D_\beta$ is invariant under the structure group Spin$^c(n)$.
Thus, globally as well, the kernel of $D_{\beta_F\otimes \pi^*E}$ identies with the kernel of $D_E$.

Similarly, using 
\[(D_1\# D_2)^* = -D_1^*\# D_2\]
the kernel of the formal adjoint of $D_{\beta_F\otimes \pi^*E}$ identifies with the kernel of the formal adjoint of $D_E$.
Hence the cokernel  of $D_{\beta_F\otimes \pi^*E}$ identifies with the cokernel of $D_E$.
\end{proof}

\section{The topological index} \label{tind}

In this section we prove that the homomorphism of abelian groups 
\[ K_0(\cdot)\to \QQ\qquad (M,E)\mapsto (\mathrm{ch}(E)\cup \mathrm{Td}(M))[M]\]
is well-defined, i.e. that it is compatible with the three elementary moves.
For direct sum - disjoint union this is implied by $\mathrm{ch}(E\oplus F)=\mathrm{ch}(E)+\mathrm{ch}(F)$.

\subsection{Bordism invariance}

The Todd class and Chern character are stable characteristic classes.
Therefore, the cohomology class $\mathrm{ch}(E|_{\partial\Omega})\cup \mathrm{Td}(\partial\Omega)$  is the restriction to the boundary of  $\mathrm{ch}(E)\cup \mathrm{Td}(\Omega)$.
By Stokes' Theorem, the topological index of a boundary is zero.

\subsection{Vector bundle modification}\label{vbmod}

Invariance of the topological index under vector bundle modification is
\[   (\mathrm{ch}(E\otimes \beta_F)\cup \mathrm{Td}(\Sigma F))[\Sigma F]  =(\mathrm{ch}(E)\cup \mathrm{Td}(M))[M]\]
Since
\[T(\Sigma F)\oplus \underline{\RR} \cong \pi^*(TM)\oplus \pi^*F\oplus \underline{\RR}\]
multiplicativity of the Todd class implies
\[ \mathrm{Td}(\Sigma F) = \pi^*\mathrm{Td}(M)\cup \pi^*\mathrm{Td}(F)\]
Therefore invariance of the topological index under vector bundle modification 
is equivalent to
\begin{proposition}\label{prop:vbmod}
\[\pi_!\,\mathrm{ch}(\beta_F) = \frac{1}{\mathrm{Td}(F)}\]
where $\pi_!$ is integration along the fiber of $\pi\,\colon \Sigma F\to M$.
\end{proposition}
\begin{proof}
Denote by $\kappa:H^\bullet(BF,SF)\to H^\bullet(M)$ the map which is the composition
\[ H^\bullet(BF, SF)\to H^\bullet(BF,\emptyset)=H^\bullet(BF)\cong H^\bullet(M)\]
Then integration in the fiber has the property
\[ \pi_!(a)\cup\chi(F)=\kappa(a)\qquad \qquad a\in H^\bullet(BF,SF)\]
where $\chi(F)$ is the Euler class of $F$.
With $\lambda_F$ as in section \ref{Thom},
\[ \kappa(\mathrm{ch}(\lambda_F)) = \mathrm{ch}(S^+_F)- \mathrm{ch}(S^-_F)\]
The difference of Chern characters
\[ t(F) := \mathrm{ch}(S_F^+)-\mathrm{ch}(S_F^-)\in H^\bullet(M,\RR)\]
is a characteristic class of the Spin$^c$ vector bundle $F$.
A  character calculation shows that this class $t$ is
\[ t=(-1)^r e^{x/2}\prod_{j=1}^r (e^{x_j/2}-e^{-x_j/2})\]
where $x_1,\dots,x_r$ are the Pontryagin roots, and $x=c_1$.
(See \cite{Mi63} for the calculation in the Spin case. See \cite{BoSe} for the calculation in the case of a complex vector bundle). 

The Thom class $\tau_F$ is dual (conjugate) to $\lambda_F$. Replacing $x$ by $-x$ and $x_j$ by $-x_j$ we obtain
\[\kappa(\mathrm{ch}(\tau_F)) = \chi(F)\cup  \frac{1}{\mathrm{Td}(F)}\]
and therefore
\[\pi_!(\mathrm{ch}(\tau_F)) = \frac{1}{\mathrm{Td}(F)}\]
Finally, by Lemma \ref{Thom_clutch} below, $\beta_F$ is isomorphic to the vector bundle on $\Sigma F$ obtained by clutching $\tau_F$, i.e.
\begin{align*}
\Sigma F &=S(F\oplus \underline{\RR}) = \Sigma^+\cup_{S(F)} \Sigma^-\\
\beta_F&=\pi^*\overline{S_F^+}\cup_c \,  \pi^*\overline{S_F^-}
\end{align*}
which implies
\[\pi_!(\mathrm{ch}(\beta_F)) = \pi_!(\mathrm{ch}(\tau_F))\]

\end{proof}

\begin{remark}

Let $M$ and $N$ be Spin$^c$ manifolds without boundary,  not necessarily compact.
Let $f:M\to N$ be a $C^\infty$ map.
The Gysin map in compactly supported $K$-theory
\[ f_!:K^j(M)\to K^{j+\epsilon}(N)\]
(where $\epsilon=\mathrm{dim}\, M - \mathrm{dim}\, N$) is defined as follows.
Choose an embedding $\iota:M\to \RR^n$ with $n=\epsilon$ modulo 2.
Consider  the embedding
\[ (f,\iota)\;\colon\; M\hookrightarrow N\times \RR^n \qquad p\mapsto (f(p),\iota(p))\]
The normal bundle $\nu$ of this embedding $(f,\iota)$ has even dimensional fibers.
Via the short exact sequence
\[ 0\to TM\to T(N\times \RR^n)|M \to \nu\to 0\]
and the 2 out of 3 lemma (Lemma \ref{2outof3}), $\nu$ is a Spin$^c$ vector bundle.
Let $\tau$ denote the Thom class of $\nu$.
Then $f_!$ is the composition of the Thom isomorphism $K^j(M)\cong K^j(\nu)$
with excision $K^j(\nu)\to K^j(N\times \RR^n)$, and finally the K\"unneth theorem (i.e. Bott periodicity) $K^j(N\times \RR^n)\cong K^{n+\epsilon}(N)$.

The differentiable Riemann-Roch theorem states that there is commutativity in the diagram
\[ \xymatrix{   K^j(M) \ar[r]^-{f_!}\ar[d]_{\mathrm{ch}(-)\cup \mathrm{Td}(M)} & K^{j+\epsilon}(N) \ar[d]^{\mathrm{ch}(-)\cup \mathrm{Td}(N)}& \\
 H_c^\bullet(M,\QQ)\ar[r]_{f_!} & H_c^\bullet(N,\QQ)}
\]
where the lower horizontal arrow is the Gysin map in compactly supported cohomology.

Commutativity of this diagram is proven by applying the formula
\[\pi_!(\mathrm{ch}(\tau)) = \frac{1}{\mathrm{Td}(\nu)}\]
of Proposition \ref{prop:vbmod}, and is a modern version of Theorem 1 of Atiyah and Hirzebruch in \cite{AH59}.

\end{remark}

\section{Appendix: Sign conventions}

In this appendix we make explicit the sign conventions used in this paper.
Mathematicians who work in index theory are well aware that resolving sign issues can be a rather difficult problem.
In a section on sign conventions \cite{Pa63} p.281, R.\ S.\ Palais and R.\ T.\ Seeley remarked,
\begin{quote}
The definition of the topological index is replete with arbitrary choices, and a change in any one of these choices calls for a compensating multiplication by a power of $-1$.
Since this has caused so many headaches already, it is perhaps in order to make these various choices explicit and remark on the effect of a change in each of them.
\end{quote}
In \cite{APS3} p.73, Atiyah, Patodi and Singer remark
\begin{quote}
One final comment concerns the question of signs. 
There are many places where sign conventions enter and although we have endeavoured to be consistent we have not belaboured the point. 
Any error in sign in the final results will more readily be found by computing examples than by checking all the intermediate stages.
\end{quote}
For the present paper, the sign conventions we have used are {\em not} arbitrary, but rather are dictated by the underlying mathematics. 
In this spirit we now explain our choices of
\begin{itemize}
\item The grading of the spinor bundle
\item The orientation and Spin$^c$ structure of a boundary
\item The orientation and Spin$^c$ structure of a complex vector bundle
\item The Bott generator
\item The Thom class
\end{itemize}
All definitions and sign conventions involving characteristic classes are as in \citelist{\cite{MilSta} \cite{Hirz}}.
Our guiding principle is that the final index formula should be sign free.

\subsection{Orientation and Spin$^c$ structure of a boundary}
If $\Omega$ is an oriented manifold with boundary $M$,
then $M$ is oriented via the rule ``exterior normal first''.
This convention gives Stokes' Theorem free of signs \citelist{\cite{Mi65} \cite{BT}}. 

If $\Omega$ is an odd dimensional Spin$^c$ manifold with boundary $M$, 
then the spinor bundle of $M$ is the restriction of the spinor bundle of $\Omega$ to $M$,
graded by the Clifford action $ic({\bf n})$, where ${\bf n}$ is the outward pointing normal vector.
See section \ref{spincm}.
This convention is consistent with the choice of irreducible representation of the Clifford algebra $\mathrm{Cliff}(n)$ if $n$ is odd, 
implicit in the matrices $E_1,\dots, E_n$ of section \ref{Dirac_Rn}.
It  gives the correct grading of the spinor bundle on the even dimensional manifold $M$.

\subsection{Positive and negative spinors}
For a closed Spin manifold $M$ with Dirac operator $D$  
\[ \mathrm{Index}\; D = \widehat{A}(TM)[M]\]
This formula is free of signs provided that ``correct'' choices have been made for the positive and negative spinors.
In distinguishing the positive and negative spinor representations of $\mathrm{Spin}(2r)$ we are following Milnor \cite{Milnor}.
This is compatible with our choice of grading operator 
$i^rE_1E_2\cdots E_{2r}$
in section \ref{Dirac_Rn}.

\subsection{Complex structures and Spin$^c$ structures}\label{complex}
For a closed complex analytic manifold $M$ with Dirac operator $D$  
\[ \mathrm{Index}\; D = \mathrm{Td}(T^{1,0}M)[M]\]
This formula is valid if the Dirac operator is the assembled Dolbeault complex of $M$.
This is the case if a ``correct'' convention is used for the Spin$^c$ structure determined by the complex analytic structure.
To verify this, we need to review some details about Clifford algebras and the spinor representation.

As in section \ref{strgrp}
\[\Spinc = \Spin \times_{\ZZ/2\ZZ} U(1)\]
There is a 2-to-1 covering homomorphism
\begin{align*}
 \Spinc &\to \SO \times U(1)\\
 (g,\lambda) &\mapsto (\rho(g), \lambda^2)\qquad g\in \Spin,\; \lambda\in U(1)
\end{align*}
The Lie algebra representation that corresponds to the spin representation of $\Spinc$ is
\[ \Lambda^2\to M_{2^r}(\CC)\;\colon\; e_ie_j\mapsto E_iE_j\]
which is also given by
\[ \so(n)\to M_{2^r}(\CC)\;\colon\; J_{ij}\mapsto \frac{1}{2}E_iE_j\]
If $n=2r$ is even, we fix the identification of $\CC^r$ with $\RR^n$ as 
\[ \CC^r\to \RR^n\;\colon\; (z_1,\dots,z_r)\mapsto (x_1,y_1,\dots, x_r,y_r)\qquad z_j=x_j+iy_j\]
What is significant  is the implied orientation of $\CC^r$.
If $\{f_1, \dots, f_r\}$ is any basis of $\CC^r$,
the real vector space underlying  $\CC^r$ is oriented by the $\RR$ basis
\[ \{f_1, Jf_1, f_2, Jf_2, \dots , f_r, Jf_r\}\]
This convention for orienting complex vector spaces is compatible with direct sums.

The identification $\CC^r\to \RR^n$ induces an inclusion $M_r(\CC)\subset M_n(\RR)$, which restricts to a homomorphism
\[ j\;\colon\; U(r)\to SO(n)\]
The canonical homomorphism 
\[ \tilde{l}\;\colon\;U(r)\to \Spinc\]
is, by definition, the unique lift to $\Spinc$ of
\[ l\;\colon\; U(r)\to SO(n)\times U(1)\;\colon \; l(g):=j(g)\times \det{(g)}\]
For a diagonal matrix $T\in U(r)$  the canonical map to $\Spinc$ can be made explicit as follows.
(An equivalent formula for $\tilde{l}(T)$ appears, without proof, in [1] \textsection 3).

\begin{lemma}\label{lift}
The canonical inclusion $\tilde{l}\;\colon\;U(r)\to \Spinc$ maps the diagonal matrix 
$T=\mathrm{diag}(\lambda_1,\dots, \lambda_r)\in U(r)$
to the element
\[\tilde{l}(T) = \prod_{j=1}^r \left(\frac{1}{2}(1+\lambda_j)+\frac{1}{2}(1-\lambda_j)\,(ie_{2j-1}e_{2j})\right)\in \Spinc\subset C_n\otimes \CC\]
\end{lemma}

\begin{proof}
Let $\lambda_j=e^{i\theta_j}$, and consider the one-parameter subgroup
\[ \mathrm{diag}(e^{is\theta_1},\dots, e^{is\theta_r})\in U(r)\qquad s\in \RR\]
The canonical inclusion of this one parameter group in $\SO$ is
\[ j(\mathrm{diag}(e^{is\theta_1},\dots, e^{is\theta_r})) = \prod_{j=1}^r \exp(s\theta_j J_{2j-1,2j})\in \SO\]
To  lift  it to $\Spin$,
we use the isomorphism of Lie algebras $d\rho\;\colon\Lambda^2\cong \so(n)$.
The skew symmetric matrix $J_{ij}\in \so(n)$ corresponds to $\frac{1}{2}e_ie_j\in \Lambda^2\subset C_n$.
Therefore the one-parameter group lifts to
\[\prod_{j=1}^r \exp(\frac{1}{2}s\theta_j e_{2j-1}e_{2j})\in \Spin\]
On the other hand,  the determinant of the one-parameter group
\[ \det(\mathrm{diag}(e^{is\theta_1},\dots, e^{is\theta_r})) = \prod_{j=1}^r e^{is\theta_j}\in U(1)\] 
 lifts to the square root
\[ \prod_{j=1}^r e^{is\theta_j/2} \in U(1)\]
Evaluating at $s=1$ we see that  
\[ \tilde{l}(T)=
 \prod_{j=1}^r e^{i\theta_j/2}\,\exp(\frac{1}{2}\theta_j e_{2j-1}e_{2j})\in \Spinc\subset C_n\]
From $(e_ie_j)^2=-I$ (if $i\ne j$) we obtain
\[ \exp(te_ie_j)=\cos{(t)}+\sin{(t)} e_ie_j = \frac{1}{2}(e^{-it}+e^{it})+\frac{1}{2}(e^{-it}-e^{it})(ie_ie_j)\] 
and so
\[e^{i\theta_j/2}\,\exp(\frac{1}{2}\theta_j e_{2j-1}e_{2j})
=\frac{1}{2}(1+e^{i\theta_j})+\frac{1}{2}(1-e^{i\theta_j})\,(ie_{2j-1}e_{2j})\]

\end{proof}

\begin{proposition}\label{spinors}
The representation of $U(r)$ on the spinor vector bundle $\CC^{2^r}$,
obtained by composition of the canonical map $U(r)\to \Spinc$ with the spinor representation of $\Spinc$, is unitarily equivalent to the standard representation of $U(r)$ on $\Lambda^\bullet \CC^r$.
\end{proposition}
\begin{proof}
The trace of all elements $E_{i_1}E_{i_2}\cdots E_{i_p}$ with $i_1<\cdots<i_p$ is zero (for $p=1,\dots,n)$, while $\mathrm{tr}(I)=2^r$.
From Lemma \ref{lift} we see that the character $\chi$ of the spinor representation  $U(r)\to M_{2^r}(\CC)$ is
\[ \chi(\mathrm{diag}(\lambda_1,\dots, \lambda_r)) = \prod_{j=1}^r (1+\lambda_j)\]
The character of the representation of $U(r)$ on 
\[ \Lambda^\bullet \CC^r=\CC\oplus \CC^r\oplus \Lambda^2\CC^r\oplus\cdots \oplus \Lambda^r\CC^r\]
is
\[ 1 + \sum_j \lambda_j + \sum_{j_1<j_2} \lambda_{j_1}\lambda_{j_2}+\cdots + \prod \lambda_j\] 
which equals $\chi$.

\end{proof}

Recall that for $n=2r$ even, the spinor vector space $\CC^{2^r}$ splits into the $\pm 1$ eigenspaces of the matrix
\[ \omega  = i^r E_1E_2\cdots E_n\in M_{2^r}(\CC)\]

\begin{proposition}\label{grading}
Under the identification of the spinor vector bundle $\CC^{2^r}$ with the $U(r)$ representation space $\Lambda^\bullet \CC^r$,  positive spinors are identified with even forms, and negative spinors with odd forms.
\end{proposition}
\begin{proof}
Lemma \ref{lift} shows that the canonical inclusion $U(r)\to \Spinc$ maps the matrix 
\[ -I = \mathrm{diag}(-1, \dots, -1)\in U(r)\]
to
\[ \prod_{i=1}^r ie_{2j-1}e_{2j}\in \Spinc\subset C_n\otimes \CC\]
which, in the spinor representation, maps to the grading operator $\omega$.
The matrix $-I\in U(r)$ acts as $+1$ on even tensors, and as $-1$ on odd tensors.

\end{proof}

Thus, the above map $\tilde{l}\;\colon\;U(r)\to \Spinc$ determines the ``correct'' Spin$^c$ structure for a complex analytic manifold $M$.
The Dirac operator of $M$ is indeed the assembled Dolbeault complex.

\subsection{The Bott generator and the Thom class}

Let $n=2r$ be an even positive integer.
The Bott generator vector bundle $\beta$ on $S^{n}$  is the dual of the positive spinor bundle.
This choice is such that $\mathrm{ch}(\beta)[S^{n}]=1$ (see section \ref{Bottvb}),
while also $\mathrm{Index} D_\beta = 1$, which means that, when $[\beta]\in K^0(S^{n})$ is evaluated against the fundamental cycle in $K$-homology (i.e. the Dirac operator), the result is 1 (see section \ref{indexone}).

The $K$-theory group $K^0(\RR^n)$ is an infinite cyclic group.
We choose as the Bott generator of $K^0(\RR^n)$ the element that, when ``clutched'', gives the Bott generator vector bundle on $S^n$.
The clutching map
\[ c\;\colon\; K^0(\RR^n)\to K^0(S^n)\]
is defined as follows.
Let $[E,F,\sigma]\in K^0(\RR^n)$, where $E, F$ are two $\CC$ vector bundles on $\RR^n$, and  $\sigma:E\to F$ is   a vector bundle map that is an isomorphism outside a compact set. We may assume that $\sigma$ is an isomorphism on the unit sphere and outside the unit ball.

$S^n$ is the unit sphere in $\RR^{n+1}$.
Let $S^n_+$ be the upper hemisphere (with $x_{n+1}\ge 0$) and $S^n_-$ the lower hemisphere of $S^n$ (with $x_{n+1}\le 0$).
Then $c(E,F,\sigma)$ is the vector bundle on $S^n$ obtained by clutching $S^n_+\times E$ to $S^n_-\times F$ via $\sigma$,
\[ c(E,F,\sigma) = S^n_+\times E\;\cup_\sigma\;S^n_-\times F\]
For $x=(x_1,\dots,x_n)\in \RR^n$, let $\sigma(x)$ be the  matrix
\[  \sigma(x)=\sum_{j=1}^n ix_jE_j\in M_{2^r}(\CC)\]
We denote $\Delta=\CC^{2^r}$, and $\Delta=\Delta^+\oplus \Delta^-$  is the decomposition into $\pm 1$ eigenspaces of the grading operator $i^rE_1E_2\cdots E_n=iE_{n+1}$.
If $x\ne 0$, then  $\sigma(x)$ restricts to a linear isomorphism
\[ \sigma(x)\;\colon \Delta^+\to \Delta^-\]

\begin{lemma}\label{Thom_clutch}
$c(\RR^n\times \Delta^+,\RR^n\times \Delta^-,\sigma)$ is isomorphic to the positive spinor bundle of $S^n$. 
\end{lemma}

\begin{proof}
By definition, the positive spinor bundle $S^+$ of $S^n$ is the subbundle of $S^n\times \Delta$ whose fiber at $x\in S^n$ is the $+1$ eigenspace of the matrix 
\[ ic(x)=\sum_{j=1}^{n+1} ix_jE_j\in M_{2^r}(\CC) \qquad x=(x_1,\dots,x_{n+1})\in S^n\]
We use polar coordinates centered at $p_\infty$.
For $v=(v_1,\dots,v_n)\in S^{n-1}$ and $\theta\in [0,\pi]$ let 
\[ x_1 = v_1\sin{\theta},\;\dots\;, x_n=v_n\sin{\theta}, \;x_{n+1} = -\cos{\theta} \qquad x\in S^n\]
so that
\[ ic(x) =  -\cos{\theta}\cdot iE_{n+1} + \sin{\theta}\cdot ic(v)\in M_{2^r}(\CC)\]
The positive spinor bundle $S^+$ is homotopic and hence isomorphic to the vector bundle $V$ whose fiber on the upper hemisphere is $\Delta^+$, and at a point $x\in S^n_-$ of the lower hemisphere is the $+1$ eigenspace of 
\[ a(x) =  -\cos{(2\theta)}\cdot iE_{n+1} + \sin{(2\theta)}\cdot ic(v)\in M_{2^r}(\CC)\]
On the lower hemisphere define the unitary matrices
\[ g(x):=\cos{(\theta)}  + \sin{(\theta)}\cdot E_{n+1}c(v)\qquad x\in S^n_-\]
Note that $g(p_\infty)=I$, while on the equator 
$g(x)=E_{n+1}c(v)$.

Since $v\perp e_{n+1}$ we have $E_{n+1}c(v)=-c(v)E_{n+1}$ and therefore $g(x)E_{n+1}=E_{n+1}g(x)^{-1}$. We get
\begin{align*}
 a(x) 
&= -iE_{n+1}\cdot \left(\cos{(2\theta)}  + \sin{(2\theta)}\cdot E_{n+1}c(v)\right)\\
&=-iE_{n+1}\cdot g(x)^2\\
&=g(x)^{-1} \cdot (-iE_{n+1}) \cdot g(x)
\end{align*}
It follows that the linear map $g(x)\;\colon\;\Delta\to \Delta$ maps the $+1$ eigenspace of $a(x)$ to the  $+1$ eigenspace of $-iE_{n+1}$. 
The $+1$ eigenspace of $a(x)$ is the fiber of $V$, and the $+1$ eigenspaces of $-iE_{n+1}$ is $\Delta^-$.
Thus, on the lower hemisphere $g(x)$ maps the fiber of  $V$ to $\Delta^-$.
Therefore $V$ is isomorphic to the vector bundle on $S^n$ obtained by taking its restriction to the upper hemisphere, which is $V|S^n_+=S^n_+\times \Delta^+$, and clutching it to the trivial bundle $S^n_-\times \Delta^-$ on the lower  hemisphere via the isomorphism at the equator (where $\theta=\pi/2$ and $x=v$) given by
\[ g(v) = E_{n+1}c(v)\;\colon\; \Delta^+\to \Delta^- \qquad v\in S^{n-1}=S^n_+\cap S^n_-\]
$E_{n+1}$ acts as $iI$ on $\Delta^-$ and the formula simplifies to
\[ g(v) = ic(v)\;\colon\; \Delta^+\to \Delta^-\]
\end{proof}

Hence, clutching the ``obvious'' generator of $K^0(\RR^n)$ gives the positive spinor bundle of $S^n$.
Clutching the {\em conjugate} of the obvious generator gives the Bott generator vector bundle. 
In summary, the ``correct'' Bott generator $\beta_n\in K^0(\RR^n)$ is the conjugate of  the ``obvious'' choice $[\RR^n\times \Delta^+,\RR^n\times \Delta^-,\sigma]$.
\footnote{$\sigma$ is the symbol of the Dirac operator. It is not entirely surprising that the Bott generator is dual to the symbol of the Dirac operator.}
The ``correct'' Bott generator of $K^0(\RR^n)$ has the property that $\mathrm{ch}(\beta_n)[\RR^n]=1$.
Note that in this formula the Chern character $\mathrm{ch}(\beta_n)$ is an element in cohomology with compact supports.

Our choice of Thom class in section \ref{Thom} is based on the above observations.
In addition, for a complex vector bundle with its Spin$^c$ structure determined as in section \ref{complex},
our Thom class as in section \ref{Thom} is equal to the Thom class obtained via algebraic geometry as in \cite{BoSe}.

\bibliographystyle{abbrv}
\bibliography{MyBibfile}

\begin{bibdiv}
\begin{biblist}

\bib{ABS}{article}{
      author={Atiyah, M.~F.},
      author={Bott, R.},
      author={Shapiro, A.},
       title={Clifford modules},
        date={1964},
        ISSN={0040-9383},
     journal={Topology},
      volume={3},
      number={suppl. 1},
       pages={3\ndash 38},
      review={\MR{0167985}},
}

\bib{AH59}{article}{
      author={Atiyah, M.~F.},
      author={Hirzebruch, F.},
       title={Riemann-{R}och theorems for differentiable manifolds},
        date={1959},
        ISSN={0002-9904},
     journal={Bull. Amer. Math. Soc.},
      volume={65},
       pages={276\ndash 281},
      review={\MR{0110106 (22 \#989)}},
}

\bib{APS3}{article}{
      author={Atiyah, M.~F.},
      author={Patodi, V.~K.},
      author={Singer, I.~M.},
       title={Spectral asymmetry and {R}iemannian geometry. {III}},
        date={1976},
        ISSN={0305-0041},
     journal={Math. Proc. Cambridge Philos. Soc.},
      volume={79},
      number={1},
       pages={71\ndash 99},
      review={\MR{0397799}},
}

\bib{AS1}{article}{
      author={Atiyah, M.~F.},
      author={Singer, I.~M.},
       title={The index of elliptic operators. {I}},
        date={1968},
        ISSN={0003-486X},
     journal={Ann. of Math. (2)},
      volume={87},
       pages={484\ndash 530},
}

\bib{BvE2}{article}{
      author={Baum, Paul},
      author={van Erp, Erik},
       title={${K}$-homology and {F}redholm operators {II}: Elliptic
  operators},
     journal={forthcoming},
}

\bib{BvE3}{article}{
      author={Baum, Paul~F.},
      author={van Erp, Erik},
       title={{$K$}-homology and index theory on contact manifolds},
        date={2014},
     journal={Acta Math.},
      volume={213},
      number={1},
       pages={1\ndash 48},
}

\bib{BGV}{book}{
      author={Berline, Nicole},
      author={Getzler, Ezra},
      author={Vergne, Mich{\`e}le},
       title={Heat kernels and {D}irac operators},
      series={Grundlehren der Mathematischen Wissenschaften [Fundamental
  Principles of Mathematical Sciences]},
   publisher={Springer-Verlag, Berlin},
        date={1992},
      volume={298},
        ISBN={3-540-53340-0},
         url={http://dx.doi.org/10.1007/978-3-642-58088-8},
      review={\MR{1215720}},
}

\bib{BoSe}{article}{
      author={Borel, Armand},
      author={Serre, Jean-Pierre},
       title={Le th\'eor\`eme de {R}iemann-{R}och},
        date={1958},
        ISSN={0037-9484},
     journal={Bull. Soc. Math. France},
      volume={86},
       pages={97\ndash 136},
      review={\MR{0116022}},
}

\bib{BT}{book}{
      author={Bott, Raoul},
      author={Tu, Loring~W.},
       title={Differential forms in algebraic topology},
      series={Graduate Texts in Mathematics},
   publisher={Springer-Verlag, New York-Berlin},
        date={1982},
      volume={82},
        ISBN={0-387-90613-4},
      review={\MR{658304}},
}

\bib{Hi91}{article}{
      author={Higson, Nigel},
       title={A note on the cobordism invariance of the index},
        date={1991},
        ISSN={0040-9383},
     journal={Topology},
      volume={30},
      number={3},
       pages={439\ndash 443},
         url={http://dx.doi.org/10.1016/0040-9383(91)90024-X},
      review={\MR{1113688 (92f:58171)}},
}

\bib{Hirz}{book}{
      author={Hirzebruch, Friedrich},
       title={Topological methods in algebraic geometry},
      series={Classics in Mathematics},
   publisher={Springer-Verlag, Berlin},
        date={1995},
        ISBN={3-540-58663-6},
        note={Translated from the German and Appendix One by R. L. E.
  Schwarzenberger, With a preface to the third English edition by the author
  and Schwarzenberger, Appendix Two by A. Borel, Reprint of the 1978 edition},
      review={\MR{1335917}},
}

\bib{Mi63}{incollection}{
      author={Milnor, John},
       title={The representation ring of some classical groups},
        date={1963},
   booktitle={Collected papers of {J}ohn {M}ilnor: {V}. {A}lgebra},
       pages={143\ndash 154},
}

\bib{Mi65}{book}{
      author={Milnor, John~W.},
       title={Topology from the differentiable viewpoint},
      series={Based on notes by David W. Weaver},
   publisher={The University Press of Virginia, Charlottesville, Va.},
        date={1965},
      review={\MR{0226651}},
}

\bib{MilSta}{book}{
      author={Milnor, John~W.},
      author={Stasheff, James~D.},
       title={Characteristic classes},
   publisher={Princeton University Press, Princeton, N. J.; University of Tokyo
  Press, Tokyo},
        date={1974},
        note={Annals of Mathematics Studies, No. 76},
      review={\MR{0440554}},
}

\bib{Pa63}{article}{
      author={Palais, Richard~S.},
       title={Seminar on the {A}tiyah-{S}inger index theorem},
        date={1965},
       pages={x+366},
}

\end{biblist}
\end{bibdiv}





\end{document}